\documentclass [12pt]{article}   

\usepackage{a4wide}    

\usepackage[utf8]{inputenc}   
\usepackage[T1]{fontenc}            
\usepackage[english, french, english]{babel}       
\usepackage{textcomp}                
\usepackage{graphicx}       
\usepackage{listing}        

\usepackage{amsthm}
\usepackage{amsmath}
\usepackage{amssymb}
\usepackage{amsfonts}
\usepackage{mathrsfs}
\usepackage{dsfont}
\usepackage{mathabx}
\usepackage{envmath}
\usepackage{cases}

\usepackage{epic, ecltree}
\usepackage{epsfig}
\usepackage{graphics}
\usepackage{color}

\usepackage{pgf,tikz,pgfplots} 
\usepackage{mathrsfs}
\usetikzlibrary{arrows}

\usepackage{fancyhdr}     

\usepackage{lmodern}
\DeclareFontFamily{OMX}{lmex}{}
\DeclareFontShape{OMX}{lmex}{m}{n}{<->lmex10}{}

\usepackage[colorlinks=true, linkcolor=black, urlcolor=blue, filecolor=black, citecolor=black, menucolor=black]{hyperref}  

\usepackage{enumitem} 

\newcommand{\enstq}[2]{\left\lbrace#1\mathrel{}\middle|\mathrel{}#2\right\rbrace} 

\newcommand{\ensemblenombre}[1]{\mathbb{#1}} 
\newcommand{\N}{\ensemblenombre{N}}
\newcommand{\Z}{\ensemblenombre{Z}}

\newcommand{\R}{\ensemblenombre{R}}
\newcommand{\C}{\ensemblenombre{C}}
\newcommand{\T}{\ensemblenombre{T}}
\newcommand{\D}{\ensemblenombre{D}}

\newcommand{\Hol}{\operatorname{Hol}}
\newcommand{\Ran}{\operatorname{Ran}}
\newcommand{\supp}{\operatorname{supp}}

\newcommand{\abs}[1]{\left\lvert#1\right\rvert} 
\newcommand{\fonction}[5]{\begin{array}{ccccc}
#1& : & #2 & \longrightarrow & #3 \\
 & & #4 & \longmapsto & \displaystyle{#5}
\end{array}}

\makeatletter
\newcommand{\bigint}{\@ifnextchar_\@bigintsub\@bigintnosub}
\def\@bigintsub_#1{\def\@int@subscript{#1}\@ifnextchar^\@bigintsubsup\@bigintsubnosup}
\def\@bigintsubsup^#1{\mathop{\text{\large$\int_{\text{\normalsize$\scriptstyle\@int@subscript$}}^{\text{\normalsize$\scriptstyle#1$}}$}}\nolimits}
\def\@bigintsubnosup{\mathop{\text{\large$\int_{\text{\normalsize$\scriptstyle\@int@subscript$}}$}}\nolimits}
\def\@bigintnosub{\@ifnextchar^\@bigintnosubsup\@bigintnosubnosup}
\def\@bigintnosubsup^#1{\mathop{\text{\large$\int^{\text{\normalsize$\scriptstyle#1$}}$}}\nolimits}
\def\@bigintnosubnosup{\mathop{\text{\large$\int$}}\nolimits}
\makeatother

\newtheorem{thm}{Theorem}[section]
\newtheorem{prop}[thm]{Proposition}

\newtheorem{cor}[thm]{Corollaire}
\newtheorem{lemme}[thm]{Lemma}

\theoremstyle{definition} 
\newtheorem{rmk}[thm]{Remark}
\newtheorem{ex}[thm]{Exemples}

\newcounter{qu}

\newcounter{qq}
\newcounter{ex}
\title{Reachable states and holomorphic function spaces for the 1-D heat equation}
\author{Marcu-Antone \bsc{Orsoni}}
\date{\today}

\usepackage{remreset}
\makeatletter \@removefromreset{figure}{subsection}\makeatother
\makeatletter \@removefromreset{figure}{section}\makeatother

\begin{document}

\maketitle

\fancyfoot{}                 
\lhead{\leftmark}      
\chead{}                      
\rhead{\thepage }        
\abstract{The description of the reachable states of the heat equation is one of the central 
questions in control theory. The aim of this work is to present new results for the 1-D heat equation with boundary control on the segment $[0, \pi]$. In this situation it is known that the reachable states are holomorphic in a square $D$ the diagonal of which is given by $[0,\pi]$.
The most precise results obtained recently say that the reachable space
is contained between two well known spaces of analytic function: the Smirnov space $E^2(D)$ and the Bergman space $A^2(D)$.
We show that the reachable states are exactly the sum of two Bergman spaces on sectors the intersection of which is $D$. In order to get a more precise information on this sum of Bergman spaces, we also prove that it includes the Smirnov-Zygmund space $E_{L\log^+\!L}(D)$ as well as a certain weighted Bergman space on $D$.
}

\section{Introduction}
\subsection{Problem setting}
\noindent
Denote by $X=W^{-1,2}(0, \pi)$, the dual of the Sobolev space $W^{1,2}_0(0, \pi)$.
We consider the following boundary control problem of the heat equation.
\begin{equation}
\label{HE}
\tag{HE}
\left\lbrace
	\begin{aligned}
		&\frac{\partial y}{\partial t}(t,x)-		\frac{\partial^2 y}{\partial x^2} = 0  \qquad &t >0, \ x\in (0,\pi),& \\
		&y(t,0)=u_0(t),\  \ y(t,\pi)=u_\pi(t) \qquad & t >0, &\\
		&y(0,x)= f(x) \qquad &x \in (0, \pi),
	\end{aligned}
\right.
\end{equation}
For any $u\!:=\! (u_0, u_\pi) \in L^2_{\mathrm{loc}}((0, +\infty), \C^2)$ --- the so-called input (or control) function --- and $f \in X$, this equation admits a unique solution $y \in C\left((0, +\infty), X\right)$ (see \cite[Prop. 10.7.3]{TW}) defined by 
\begin{equation}
\label{3}
\forall t >0, \ y(t, \cdot)= \T_tf + \Phi_t u
\end{equation}
where $(\T_t)_{t\geq 0}$ is the Dirichlet Laplacian semigroup and $\Phi_t$ is the controllabililty operator (see \cite[Prop. 4.2.5]{TW}). For $f \in X$ and $\tau >0$, we will say that $g \in X$ is reachable from $f$ in time $\tau$ if there exists a boundary control $u \in L^2((0,\tau), \C^2)$ such that the solution of $\eqref{HE}$ satisfies $y(\tau, \cdot)=g$. We denote by $\mathcal{R}^f_\tau$ the set of all reachable fonctions from $f$ in time $\tau$. 
Because of the smoothing effect of the heat kernel, it is clear that for arbitrary control $u \in L^2((0,\tau), \C^2)$, we cannot reach any non-regular functions. So, $\mathcal{R}^f_\tau \subsetneq X$. In other words, the equation \eqref{HE} is not exactly controllable for any time $\tau >0$.
It is thus natural to seek for more precise information on $\mathcal{R}^f_\tau$.

First of all, we remind that this set has some invariance properties. Indeed, \cite{Eg} and \cite{FR} have shown that the heat equation \eqref{HE} is null-controllable in any time (see \cite{LR}, \cite{FI} or \cite[Prop. 11.5.4]{TW} for the $n$-dimensional case). This means that : $\forall f \in X, \forall \tau>0, \  0 \in \mathcal{R}^f_\tau$. 
It is clear from \eqref{3} that this latter condition is equivalent to $\Ran \T_\tau \subset \Ran \Phi_\tau $. So using \eqref{3} again we obtain $\mathcal{R}^f_\tau= \Ran \Phi_\tau $, which means that $\mathcal{R}^f_\tau$ does not depend on the initial condition $f \in X$. Thus, we can take $f=0$ and write $\mathcal{R}_\tau=\mathcal{R}_\tau^0$. Moreover, since $\Phi_\tau \in \mathcal{L}(L^2([0,\tau], \C^2),X)$, $\mathcal{R}_\tau$ is a linear space, named \emph{reachable space} of \eqref{HE}.
Finally, the null-controllability in any time $\tau >0$ implies also that this space does not depend on time $\tau >0$ (see \cite{Fa}, \cite{Se}, or \cite[Rmk 1.1]{HKT}). Note that these two invariance properties hold for every linear control system wich satisfies the null-controllability in any positive time. 

\subsection{Notations}
\label{subsection 3}
In the rest of the paper, we denote by $\D=\enstq{z \in \C}{|z| < 1}$ the unit disc and by $\C^+=\enstq{z \in \C}{\mathrm{Im}z>0}$ the upper-half plane. 
Let $\Omega$ be a simply connected domain in the complex plane with at least two boundary points. Write $\Hol(\Omega)$ for the algebra of holomorphic functions on $\Omega$. We say that $f \in \Hol(\Omega)$  belongs to the Hardy space $H^p(\Omega)$ $(0 < p < + \infty)$ if the subharmonic function $|f|^p$ admits a harmonic majorant on $\Omega$. We say that $f\in \Hol(\Omega)$ belongs to the Smirnov space $E^p\left(\Omega \right)$ $(0 < p < + \infty)$ if there exists a sequence $(\gamma_n)_{n \in \N}$ of rectifiable Jordan curves eventually surrounding each compact subdomain of $\Omega$ such that 
$$
 \|f\|_{E^p}^p=\sup_{n} \int_{\gamma_n} |f(z)|^p |dz| < \infty.
$$ 
For $\Omega = \D$, it is well-known that these two last spaces coincide, and we will simply denote them by $H^p$. For an arbitrary conformal mapping $\varphi : \mathbb{D} \to \Omega$, $f \in H^p(\Omega)$ if and only if $f \circ \varphi \in H^p$, and $f \in E^p\left(\Omega \right)$ if and only if $(f \circ \varphi) \varphi^{1/p} \in H^p$. When $\Omega = \C^+$, the space $E^p(\Omega)$ consists of all the functions holomorphic on $\C^+$ such that 
$$
\|f\|_p^p=\sup_{y>0} \int_\R |f(x+iy)|^p dx < \infty. 
$$
This space is often called the Hardy space of the upper-half plane (\cite{Ga}, \cite{Ni1}, \cite{Le}, \cite{Ko}), and denoted by $H^p(\C^+)$.
Assume now that $\Omega$ is a domain bounded by a rectifiable Jordan curve $\gamma$. In this case, each $f \in E^p(\Omega)$ ($1 \leq p < \infty$) admits a non-tangential limit almost everywhere on $\gamma$ (denoted again by $f$) which belongs to $L^p(\partial \Omega)$, and satisfies the Cauchy formula 
$$ \forall z \in \Omega, \ f(z)=\frac{1}{2i\pi} \int_{\gamma} \frac{f(u)}{u-z} du.$$
With this in mind, we will say that $f \in \Hol(\Omega)$ belongs to the Smirnov-Zygmund space $E_{L\log^{+}\!L}\left(\Omega\right)$ if $f \in E^1(\Omega)$ and its non-tangential limit on $\gamma$ belongs to $L\log^{+}\!L(\partial \Omega)$, that is $\int_{\gamma} |f(z)| \log^{+}(|f(z)|) |dz| < + \infty$.
Therefore, the following inclusions are clear 
$$\forall 1 < p < +\infty, \quad E^p(\Omega) \subset E_{L \log^+  \!\! L}(\Omega) \subset E^1(\Omega). $$
For more details on the Hardy space and the Smirnov space, we refer to \cite[Chap. 9 et 10]{Du}. For the cases of the disc and the upper-half plane, see also \cite{Ru},\cite{Ga}, \cite{Ni1}, \cite{Le}, \cite{Ko}. 
 
Finally, the weighted Bergman space $A^p\left(\Omega, \omega \right)$, where $\omega$ is a non negative mesurable function on $\Omega$, consists of all functions $f \in \Hol(\Omega)$ such that $$
\|f\|_{A^p(\Omega,\omega)}^p=\int_{\Omega} |f(x+iy)|^p \omega(x+iy)dxdy < +\infty. 
$$
When $\omega =1$, $A^p(\Omega,\omega)$ is the classical Bergman space which we simply denote $A^p\left(\Omega \right)$.  

\subsection{Scientific context}
It seems that the first work on this problem is in \cite{FR}. Using a moment method, Fattorini and Russel showed that if there exists $A >0$ and $B>0$ such that
\begin{equation}
\label{coef}
\forall n \in \N^*, \ |a_n|\leq A\exp(-(\pi+B)n)
\end{equation}
then the function defined by $g(x)=\sum_{n=1}^{+\infty} a_n \sin(nx)$ is reachable.
For $\delta >0$, denote by $H_\delta$ the space of continuous functions which are $\pi$-periodic on $\R$, which extend holomorphically on the strip $|\mathrm{Im}z|< \delta$ and whose derivatives of even orders vanish in $0$ and in $\pi$. 
Adjusting a discrete Paley-Wiener theorem\cite[Chap.IV, sect.V, Thm V.1 vi), p. 98]{QZ} to the orthonormal basis $(\sin(nx))_{n \geq 1}$, we obtain from \eqref{coef} that for $\delta$ large enough, $H_\delta$ is included in $\Ran \Phi$. 

Later, Martin, Rosier and Rouchon improved this result in \cite[Thm 1.1]{MRR}. On the one hand, they showed that the holomorphic functions on the disk 
$$
B=\enstq{z \in \C}{\abs{z-\frac{\pi}{2}} < \frac{\pi}{2} e^{(2e)^{-1}}}$$
are reachable. On the other hand, they proved that the reachable functions extend holomorphically to the square (see Figure \ref{figure1})
$$D=\enstq{z=x+iy \in \C}{\abs{x-\frac{\pi}{2}} + \abs{y} < \frac{\pi}{2}}. $$
To summarize, we have $\Hol(B) \subset \mathcal{R}_\tau \subset \Hol(D)$. 
Dardé and Ervedoza improved again this latter result in \cite[Thm 1.1]{DE} showing that all the functions which are holomorphic on a neighborhood of $D$ are reachable. This result combined with the previous result means that $\Ran \Phi_\tau$ is a space of holomorphic functions on $D$. 

Finally, the best known result on this problem to our knowledge is given in \cite{HKT}, where the authors proved  that the reachable space is sandwiched between two Hilbert spaces of holomorphic functions on the square $D$. More explicitly, it satisfies the inclusions 
\begin{equation}\label{HKT-incl}
 E^2(D) \subsetneq \Ran \Phi_\tau \subset A^2(D)
\end{equation}
(see the previous subsection for the definitions).
Key tools used in that paper include a unitary Laplace type integral operator studied by Aikawa, Hayashi and Saitoh \cite{AHS}, as well as
a Riesz basis of exponentials in $E^2(D)$ discussed by Levin and Lyubarskii \cite{LL}. 
The idea of our paper is to avoid the Riesz basis of exponentials and to use complex
analysis tools like Cauchy formula, Hilbert transform and $\overline{\partial}$-methods
which will allow us to improve significantly --- in terms of function spaces of complex analysis --- the lower bound $E^2(D)$.

\subsection{Main results}
Let $\Delta = \enstq{z \in \C}{|\arg(z)| < \frac{\pi}{4}}$. 
The first central result of this paper is the following explicit characterization of the reachable space.
\begin{thm}
\label{thm2}
We have $\Ran \Phi_\tau = A^2(\Delta) + A^2(\pi - \Delta)$.
\end{thm}

We mention another characterization which was very recently observed by T.\ Normand (see\cite{Tu}).
Denote by $\omega_0$ and $\omega_\pi$ the weights defined by 
\begin{equation}
\label{2} 
\forall z \in \Delta, \ \omega_0(z)= \frac{e^{\frac{\mathrm{Re}(z^2)}{2\tau}}}{\tau} \qquad \text{and } \qquad  \forall z \in \pi- \Delta, \ \omega_\pi(z)= \omega_0(\pi-z),
\end{equation}
then
\begin{equation}\label{Thomas}
\Ran \Phi_\tau= A^2(\Delta, \omega_0)+ A^2(\pi - \Delta, \omega_\pi),
\end{equation}
independently of $\tau>0$. Note that the inclusion ``$\supset$''  in \eqref{Thomas} was already known in \cite{HKT}.

To prove the Theorem \ref{thm2}, the main idea is to write a certain integral operator as a Laplace type transform and to use a Paley-Wiener type theorem. 
%
%

The central question raised by the above result is the description of the 
sum $A^2(\Delta) + A^2(\pi - \Delta)$. We obviously have
\begin{equation}\label{inclusion1}
A^2(\Delta) + A^2(\pi - \Delta)\subset A^2(D),
\end{equation}
and from the results in \cite{HKT} it also follows that 
$$
E^2(D)\subset A^2(\Delta) + A^2(\pi - \Delta),
$$
but how can we decide in general whether a given holomorphic function on $D$, not necessarily in $E^2(D)$, can be written as 
a sum of two functions in the Bergman spaces on $\Delta$ and $\pi-\Delta$ respectively? 
Note that this is a very natural complex analysis question which can now be discussed completely disconnected from the initial control problem. 
Such a type of problem is related, for instance, to the so-called First Cousin Problem (see \cite[Thm 9.4.1]{AM}), which in our situation turns into a specific Cousin Problem with $L^2$-estimates. We will call this problem the First Cousin Problem for Bergman spaces. Following the proof of the classical First Cousin Problem of \cite{AM} and using Hörmander $L^2$-estimates for the $\bar{\partial}$-equation, we prove the following partial answer to the description of  $A^2(\Delta) + A^2(\pi - \Delta)$.
\begin{thm}
\label{thm4}
Let $z_0=\frac{\pi}{2} + i \frac{\pi}{2}$ and $z_1=\overline{z_0}=\frac{\pi}{2} - i \frac{\pi}{2}$ be the upper  and lower vertices of $D$. Then 
$A^2\left(D, |(z-z_0)(z-z_1)|^{-2}\right) \subset A^2(\Delta) + A^2(\pi - \Delta)$. 
\end{thm}

In view of \eqref{inclusion1}
Theorem \ref{thm4} gives an optimal result for functions in the sum of the two Bergman spaces outside $z_0$ and $z_1$. However the constraints in $z_0$ and $z_1$ are rather strong, implying in particular that functions
in $A^2\left(D, |(z-z_0)(z-z_1)|^{-2}\right)$ are locally bounded at these points.

Our second result, involving completely different tools, allows to show that functions with
almost characteristic Bergman space growth, in particular at $z_0$ and $z_1$, are also in the sum.
\begin{thm}
\label{thm1}
$E_{L\log^+\!L}(D) \subset A^2(\Delta) + A^2(\pi - \Delta)$.
\end{thm}
Note that in view of Theorem \ref{thm2}, this result improves the left inclusion \eqref{HKT-incl} obtained in \cite{HKT} since $E^2(D) \subset E_{L\log^+\!L}(D)$.
The methods used in the proof of Theorem \ref{thm1} are for the most part harmonic and complex analysis methods. More precisely, we use essentially the Cauchy formula for Smirnov functions, a local regularity result for the Cauchy Transform on the upper-half plane and the embedding $H^1(\D) \subset A^2(\D)$ due to Hardy and Littlewood. 

It should be observed that in terms of growth of functions in the different spaces, this result is almost sharp (especially close to the points $z_0$ and $z_1$ which are not covered by Theorem \ref{thm4}). Indeed, a function in $A^2(\D)$ cannot grow faster than $\frac{1}{d(z, \partial D)}$, while the growth of a function in $E_{L\log^+\!L}(D)$ is bounded by $\frac{1}{d(z, \partial D)\log(\frac{1}{d(z,\partial D)})}$. This last estimate comes from the Cauchy formula and the Hölder inequality for  Orlicz spaces.

As an easy consequence of the above results we would like to state the following corollary.

\begin{cor}
We have
$$
 E_{L\log^+\!L}(D) + A^2\left(D, |(z-z_0)(z-z_1)|^{-2}\right) 
 \subset A^2(\Delta) + A^2(\pi - \Delta)\subset A^2(D). 
$$
\end{cor}

In view of Theorem \ref{thm2}, this corollary thus improves \eqref{HKT-incl} found in \cite{HKT}. Nevertheless, though $A^2\left(D, |(z-z_0)(z-z_1)|^{-2}\right)$-functions behave like arbitrary $A^2(D)$-functions outside $z_0,z_1$, and $E_{L\log^+\!L}(D)$-functions allow a growth in a sense close to $A^2(D)$-functions in $z_0,z_1$, the above corollary still leaves of course a gap. Note that there is no natural completion of $A^2\left(D, |(z-z_0)(z-z_1)|^{-2}\right)$  in $A^2(D)$.
\\

The paper is organized as follows. In Section \ref{section2}, we prove Theorem \ref{thm2}. Sections \ref{section3} and \ref{subsection3.2} are devoted to the proofs of Theorems \ref{thm1} and \ref{thm4} respectively. 

\section{Sum of Bergman spaces \label{section2}}

We start recalling some facts from \cite{HKT}. First, it is not difficult to check that the 
reachable states of the 1-D heat equation can be represented as a Fourier-sine series in the following way:
\begin{multline*}
(\Phi_\tau u)(x)=\frac{2}{\pi}\sum_{n\geqslant 1}n\left[\int_0^\tau {\rm e}^{n^2(\sigma-\tau)}  u_0(\sigma) \, d\sigma\right]\sin(n x) \\
+ \frac{2}{\pi}\sum_{n\geqslant 1}n(-1)^{n+1}\left[\int_0^{\tau} {\rm e}^{n^2(\sigma-\tau)}  u_\pi(\sigma) \, d\sigma\right]\sin(n x),
\quad \tau>0,\ x\in (0,\pi).
\end{multline*}
With the elementary formula $\sin u=(e^{iu}-e^{-iu})/(2i)$ in mind and the Poisson
summation formula, the authors of \cite{HKT} show that
\begin{equation*}
(\Phi_\tau u)(x)=\int_0^\tau  \frac{\partial K_0}{\partial x} (\tau-\sigma,x) u_0 (\sigma) \, d\sigma + \int_0^\tau \frac{\partial K_\pi}{\partial x}(\tau-\sigma,x) u_\pi (\sigma) \, d\sigma,
\end{equation*}
where
\begin{eqnarray*}
K_0 (\sigma,x)&=& -\frac{2}{\pi}\left(
\sum_{n\geqslant 1} {\rm e}^{-n^2\sigma}\cos(n x) +1\right) \\ 
 &=& -
\frac{1}{\pi} \sum_{n\geqslant 1}  {\rm e}^{-n^2 \sigma}\left({\rm e}^{i n x}+
{\rm e}^{-i n x}\right)-\frac{2}{\pi}\\
&=& - \frac{1}{\pi}   \sum_{n\in \mathbb{Z}^*} {\rm e}^{-n^2\sigma}{\rm e}^{i n x}
 -\frac{2}{\pi}, \quad \sigma>0,\ \ x\in (0,\pi).
\end{eqnarray*}

Hence, setting $\widetilde{K_0}(\sigma, z)= - \sqrt{\frac{1}{\pi \sigma}} \sum_{m \in \Z^*} e^{-\frac{(z+2m\pi)^2}{4\sigma}}$, we can write (see \cite[equation (2.18)]{HKT})
\begin{equation}
\label{eq1}
\Phi_\tau (u_0, u_\pi) = \widetilde{\Phi}_\tau u_0 + \widetilde{\widetilde{\Phi}} u_\pi + R_{0, \tau} u_0 + R_{\pi, \tau} u_\pi
\end{equation} where 
\begin{align*}
&\left[\widetilde{\Phi}_\tau f\right](s) = \int_0^\tau \frac{s e^{- \frac{s^2}{4(\tau-\sigma)}}}{2\sqrt{\pi} {(\tau-\sigma)}^{\frac{3}{2}}} f(\sigma) d\sigma \qquad \text{ and } \qquad \left[\widetilde{\widetilde{\Phi}}_\tau f\right](s) = \left[\widetilde{\Phi}_\tau f\right](\pi-s) \\
& \left[R_{0, \tau} f \right](s)= \int_0^\tau \frac{\partial \widetilde{K_0}}{\partial s}(\tau-\sigma, s) f(\sigma) d\sigma  \qquad \text{ and } \qquad \left[R_{\pi, \tau} f \right](s) = \left[R_{0, \tau} f \right](\pi-s).
\end{align*}
We can now prove Theorem \ref{thm2}. 
\begin{proof}[Proof of Theorem \ref{thm2}]
The key of the proof is that $\widetilde{\Phi}_\tau$ is an isometry from $L^2(0, \tau)$ to $A^2(\Delta)$ and we can compute its range. Indeed, denote by $\mathcal{L}$ the normalized Laplace transform defined by $\mathcal{L}(f)(s)=\frac{1}{\sqrt{\pi}}\int_0^{+\infty} e^{-st} f(t) dt$ and $G: A^2(\C_+) \to A^2(\Delta)$ the unitary operator associated to the conformal mapping $z \mapsto z^2$ from $\Delta$ to $\C_+$, defined by $G(f)(z)=2zf\left(z^2\right)$.

By the change of variables $t=\frac{1}{4(\tau-\sigma)}$, we obtain 
$$ 
\forall s \in \Delta, \ \left(\widetilde{\Phi}_\tau f\right)(s) = \int_0^\tau \frac{s e^{- \frac{s^2}{4(\tau-\sigma)}}}{2\sqrt{\pi} {(\tau-\sigma)}^{\frac{3}{2}}} f(\sigma) d\sigma 
= \frac{s}{\sqrt{\pi}} \int_\frac{1}{4\tau}^{+\infty} \frac{ e^{- s^2t}}{\sqrt{t}} \, f\!\left(\tau - \frac{1}{4t}\right) dt 
$$
Define for $f \in L^2(0, \tau)$,
$$(Tf)(t)= 
\begin{cases} \frac{f(\tau-\frac{1}{4t})}{2\sqrt{t}} \quad \text{if } t > \frac{1}{4\tau}, \\
0 \qquad \quad  \text{if } 0 < t \leq \frac{1}{4\tau}.
\end{cases}.$$ 
It is easily seen that the operator $T$ is an isometry from $ L^2(0, \tau)$ to $L^2(\R^+, \frac{dt}{t})$ with range $L^2\left(\left(\frac{1}{4\tau}, + \infty\right), \frac{dt}{t}\right)$.
Hence
$$
\left(\widetilde{\Phi}_\tau f\right)(s) = 2s \mathcal{L}(Tf)\left(s^2\right)\\
= \left(G \mathcal{L} T f \right) (s).
$$
The last step is the following Paley-Wiener type theorem for Bergman spaces (which seems to be a ``folk theorem'', a proof for which may be found in \cite{DGGMR}). 
\begin{prop}
The Laplace transform $\mathcal{L}$ is unitary from $L^2\left(\R_+, \frac{dt}{t} \right)$ to $A^2\left(\C_+\right)$ where $\C_+=\{z \mid \mathrm{Re}z >0 \}$.
\end{prop}
So, if $\simeq$ means that the operator is unitary, we have the following diagram. 
$$ L^2(0, \tau) \overset{T}{\underset{\simeq}{\longrightarrow}} L^2\left(\left(\frac{1}{4\tau}, + \infty \right), \frac{dt}{t}\right) \subset  L^2\left(\R_+, \frac{dt}{t} \right) \overset{\mathcal{L}}{\underset{\simeq}{\longrightarrow}} A^2\left(\C_+\right) \overset{G}{\underset{\simeq}{\longrightarrow}} A^2\left(\Delta\right)
$$
Hence, by composition, $\widetilde{\Phi}_\tau$ is isometric from  $L^2(0, \tau)$ to $A^2\left(\Delta\right)$, and 
$$
\Ran \widetilde{\Phi}_{\tau}= 
 G\mathcal{L}\left[L^2\left(\left(\frac{1}{4\tau}, + \infty \right), \frac{dt}{t}\right)\right]
 \subset A^2(\Delta).
$$ 
We get a similar result for $\widetilde{\widetilde{\Phi}}_{\tau}$. In order to discuss
the range of $\Phi_{\tau}$ we thus have to investigate the remainder terms
$R_{0,\tau}$ and $R_{\pi,\tau}$, which, morally speaking, are sums converging very quickly
since they involve gaussians centered essentially at $\pi n$, $n\in \Z^*$. For these
remainder terms we will use
the lemma below which is a straightforward modification of Lemma 4.1 of \cite{HKT}, the main difference being a square root in the integral operator, which does not change the boundedness and the convergence to zero. 

\begin{lemme}\label{LemmaHKT4.1}
Let $\omega_0$ and $\omega_\pi$ be the weights defined in \eqref{2}. Then $R_{0, \tau}$ and $R_{\pi, \tau}$ are bounded from $L^2(0, \tau)$ to $A^2(\Delta, \omega_0 ) + A^2(\pi-{\Delta}, \omega_\pi)$. 
Moreover, $\|R_{0, \tau}\|=\|R_{\pi, \tau}\| \underset{\tau \to 0}{\to} 0$. 
\end{lemme}

Since $A^2(\Delta, \omega_0 ) + A^2(\pi-{\Delta}, \omega_\pi)\subset A^2(\Delta)
+A^2(\pi-\Delta)$, the inclusion $\Ran\Phi_{\tau}\subset A^2(\Delta)+A^2(\pi-\Delta)$ is a direct consequence of the decomposition \eqref{eq1}, the above discussion and Lemma \ref{LemmaHKT4.1}.  
\\

For the converse inclusion, we will prove $A^2(\Delta) \subset  \Ran\Phi_\tau$ and $A^2(\pi-\Delta) \subset  \Ran\Phi_\tau$. Using that $G$ and $\mathcal{L}$ are unitary, we have the decomposition 
\begin{align*}
A^2(\Delta) &= G \mathcal{L} \left[L^2\left(\R_+, \frac{dt}{t} \right)\right] \\
&= G \mathcal{L} \left[ L^2\left(\left(0, \frac{1}{4\tau}\right), \frac{dt}{t}\right) {\oplus} L^2\left(\left(\frac{1}{4\tau}, + \infty \right), \frac{dt}{t}\right)  \right] \\
&= X_0 {\oplus} \Ran\widetilde{\Phi}_\tau
\end{align*}
where we wrote $X_0 :=  G \mathcal{L} \left[ L^2\left(\left(0, \frac{1}{4\tau}\right), \frac{dt}{t}\right) \right]$ and where, as usual ${\oplus}$ means orthogonal sum. Similarly, we have $A^2(\pi-\Delta) = X_\pi {\oplus} \Ran\widetilde{\widetilde{\Phi}}_\tau$, where $X_\pi$ is the image of $X_0$ by the transformation $f\mapsto f(\pi - \cdot)$. Hence, it is enough to prove that $X_0, X_\pi,  \Ran\widetilde{\Phi}_\tau$ and  $\Ran\widetilde{\widetilde{\Phi}}_\tau$ are contained in $\Ran\Phi_\tau$. For this, note that for every $u_0 \in L^2(0, \tau)$, we have 
$$ \widetilde{\Phi}_\tau u_0 = \Phi_\tau (u_0, 0) - R_{0, \tau}u_0 .$$
Since $A(\Delta,\omega_0) + A^2(\pi-{\Delta}, \omega_\pi) \subset \Ran \Phi_{\tau}$, we get from Lemma \ref{LemmaHKT4.1} that $R_{0, \tau}u_0 \in \Ran\Phi_\tau$. It follows that $ \Ran\widetilde{\Phi}_\tau \subset \Ran\Phi_\tau$. The case of $\Ran\widetilde{\widetilde{\Phi}}_\tau$ is similar. Finally, for $a>0$, denote by 
$$\mathrm{PW}_a(\C_+)=\enstq{f \in \mathrm{Hol}(\C)}{\exists C >0, \abs{f(z)} \leq C e^{\pi \abs{z}} \, \text{ and } \, \int_\R \abs{f(iy)}^2 dy < \infty}$$ 
the Paley-Wiener space on the right-half plane. Then by the classical Paley-Wiener theorem, $X_0 \subset G \mathcal{L} \left[ L^2\left(0, \frac{1}{4\tau}\right) \right] \subset G \mathcal{L} \left[ L^2\left(-\frac{1}{4\tau}, \frac{1}{4\tau}\right)\right] = G\left[ \mathrm{PW}_a(\C_+) \right]$. Thus, $X_0$ is a space of entire functions and, as such, is contained in the reachable space. The same argument proves also that the reachable space includes $X_\pi$, and the proof is complete. 
\end{proof}

\section{Proof of Theorem \ref{thm1}\label{section3}}
Let $P(z)=z+2i\pi$.
It suffices to prove the following assertion. 
\begin{equation}
\label{eq2}
\forall f \in E_{L\log^{+}\!L}\left(D\right), \ \frac{f}{P} \in A^2(\Delta) + A^2(\pi - \Delta)
\end{equation}
Indeed, assume that \eqref{eq2} is true and let $g \in E_{L\log^{+}\!L}\left(D\right)$. Since $P$ is bounded analytic on $\widebar{D}$, $Pg$ belongs also to $E_{L\log^{+}\!L}\left(D\right)$. Hence, by \eqref{eq2}, $g=(Pg)/P$ belongs to $A^2(\Delta) + A^2(\pi - \Delta)$ and then to $\Ran \Phi_\tau$ by Theorem \ref{thm2}, which proves the inclusion. 
\begin{rmk}
With a more refined argument, as used in \cite[corollary 3.6]{HKT} and here in Section \ref{subsection3.2}, we can prove that $E_{L\log^{+}\!L}\left(D\right) \subset A^2(\Delta, \omega_0) + A^2(\pi - \Delta, \omega_\pi)$ where $\omega_0$ and $\omega_\pi$ are defined in \eqref{2}. This observation also follows from Thm \ref{thm1} and 
Normand's result \eqref{Thomas} (see \cite{Tu})
\end{rmk}

So, pick $f \in E_{L\log^{+}\!L}\left(D\right)$ and let us prove \eqref{eq2}. 
\paragraph{Decomposition.}
Let $\gamma$ be the boundary of $D$ parameterized counterclockwise side by side as follows (see Figure \ref{figure1}). 
\begin{align*}
&\fonction{\gamma_{1,+}}{\left[0, 1\right]}{\C}{t}{(1-t)\frac{\pi}{2}(1+i)} &&\fonction{\gamma_{1,-}}{\left[0, 1 \right]}{\C}{t}{(1-i)\frac{\pi}{2}t} \\
&\fonction{\gamma_{2,+}}{\left[0, 1\right]}{\C}{t}{\pi (1-t) + t\left((1+i)\frac{\pi}{2} \right)}  &&\fonction{\gamma_{2,-}}{\left[0, 1\right]}{\C}{t}{(1-i)\frac{\pi}{2}(1-t)+t\pi} 
\end{align*}

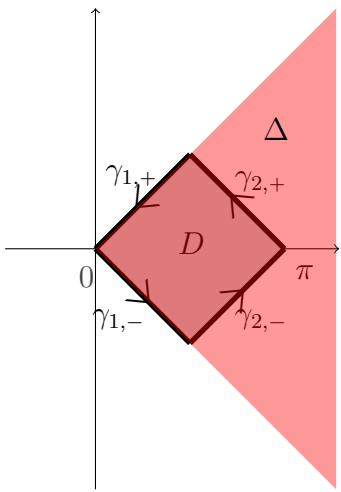
\begin{figure}[h]
\begin{center}
\definecolor{uuuuuu}{rgb}{0.26666666666666666,0.26666666666666666,0.26666666666666666}
\begin{tikzpicture}[scale=0.4]

\draw  [->](6.283185307179586,0)-- (8.1,0);
\draw  (-3,0)-- (0,0);
\draw  [->](0,-8)-- (0,8);

\fill[line width=2.pt,fill=black,fill opacity=0.2] (0.,0.) -- (3.141592653589793,-3.141592653589793) -- (6.283185307179586,0.) -- (3.141592653589794,3.141592653589793) -- cycle;

\draw [line width=2.pt] (0.,0.)-- (3.141592653589793,-3.141592653589793) node[midway,sloped,xscale=1, yscale=2]{$>$};
\draw [line width=2.pt] (3.141592653589793,-3.141592653589793)-- (6.283185307179586,0.)node[midway,sloped,xscale=1, yscale=2]{$>$};
\draw [line width=2.pt] (6.283185307179586,0.)-- (3.141592653589794,3.141592653589793)node[midway,sloped,xscale=1, yscale=2]{$<$};
\draw [line width=2.pt] (3.141592653589794,3.141592653589793)-- (0.,0.)node[midway,sloped,xscale=1, yscale=2]{$<$};

\draw[color=uuuuuu] (-0.3,-0.25) node[below] {$0$};
\draw[color=black] (3.19,0.17) node {$D$};
\draw[color=black] (2,-1.65) node[below left] {$\gamma_{1,-}$};
\draw[color=black] (5.5,-2.4) node {$\gamma_{2,-}$};
\draw[color=black] (6.283185307179586,-0.2) node[below right] {$\pi$};
\draw[color=black] (5.5,2.2) node {$\gamma_{2,+}$};
\draw[color=black] (1.2,2.4) node {$\gamma_{1,+}$};

\fill[line width=2.pt,fill=red,fill opacity=0.4] (8,8) -- (0,0) -- (8,-8)--cycle;
\draw[color=black] (6,4) node[xscale=1, yscale=1] {$\Delta$};

\end{tikzpicture}
\caption{\label{figure1}The square $D$, the path $\gamma$ and the  sector $\Delta$.}
\end{center}
\end{figure}


The key idea is to decompose $f$ {\it via} the Cauchy formula for functions in $E^1(D)$ (see \cite[Thm. 10.4 p.170]{Du}) : 
\begin{align*}
\forall z \in D, \  f(z)&= \frac{1}{2i\pi} \int_{\gamma} \frac{f(u)}{u-z} du\\
&= \frac{1}{2i\pi} \sum_{\substack{k \in \left\lbrace 1, \, 2 \right\rbrace \\ \varepsilon \in \lbrace \pm \rbrace}} \int_{\gamma_{k,\varepsilon}} \frac{f(u)}{u-z} du \\
&=\frac{1}{2} \sum_{\substack{k \in \left\lbrace 1, \, 2 \right\rbrace \\ \varepsilon \in \lbrace \pm \rbrace}} f_{k,\varepsilon}(z)
\end{align*}
where we have written 
$$
f_{k,\varepsilon}(z)=\frac{1}{i\pi}\int_{\gamma_{k,\varepsilon}} \frac{f(u)}{u-z} du,
\quad k \in \left\lbrace 1, \, 2 \right\rbrace,\ \varepsilon \in \lbrace \pm \rbrace.
$$  
For the reader acquainted with Hardy spaces, the crucial observation here is that $f_{k,\varepsilon}$ can be seen --- modulo rotation and translation --- as a scalar product 
between a (compactly supported) function and a reproducing kernel of the Hardy space, which thus yields a (Riesz-) projection on the Hardy space. It is known that this projection
is bounded when $f_{k,\varepsilon}\in L^p$, $p>1$, but
not when $p=1$. As will be explained below, on the real line, this boundedness remains valid
when $f_{k,\varepsilon}\in L\log L$ and $f_{k,\varepsilon}$ is compactly supported. 
Once we have established this fact, a theorem by Hardy and Littlewood on inclusion between
Hardy and Bergman spaces will allow to conclude.\\

The remainder part of the section will be devoted to
show that $f_{1, \varepsilon}/ P \in A^2\left(\Delta\right)$ and $f_{2, \varepsilon}/P \in A^2\left(\pi-\Delta\right)$ for $\varepsilon \in \lbrace \pm \rbrace$. 
We cut each sector $\Delta$ and $\pi - \Delta$ in two disjoint parts, which will be treated separately.  
For that, given a fixed $a >0$, denote by $D_a$ the homothetic dilation of $D$ with center 0 and obtained by adding length $a >0$ to the sides of $D$ (see Figure \ref{figure3}).  
We will consider the disjoint union $\Delta=D_a\cup \Delta\setminus D_a$ (and
similarly for $\pi-\Delta$).
The proof is composed of two steps. \\

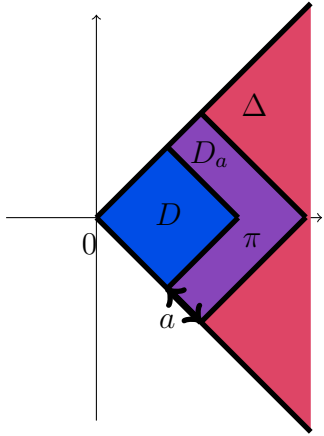
\begin{figure}[h]
\begin{center}
\definecolor{ptc}{rgb}{0,0.3,0.9}
\definecolor{gc}{rgb}{0.53,0.27,0.73}
\definecolor{sec}{rgb}{0.87, 0.27,0.4}
\begin{tikzpicture}[scale=0.3]
\draw  [->](6.283185307179586,0)-- (10,0);
\draw  (-4,0)-- (0,0);
\draw  [->](0,-9)-- (0,9);
\fill[line width=2.pt,fill=sec] (9.5,9.5) -- (0,0) -- (9.5,-9.5)--cycle;
\fill[line width=2.pt,color=blue,fill=gc] (0,0.) -- (4.635238686543314,-4.635238686543314) -- (9.283185307179584,0.016814692820412837) -- (4.635238686543314,4.635238686543314) -- cycle;
\fill[line width=2.pt,fill=ptc] (0.,0.) -- (3.141592653589793,-3.141592653589793) -- (6.283185307179586,0.) -- (3.141592653589794,3.141592653589793) -- cycle;
\draw [line width=2.pt] (3.141592653589793,-3.141592653589793)-- (6.283185307179586,0.);
\draw [line width=2.pt] (6.283185307179586,0.)-- (3.141592653589794,3.141592653589793);
\draw [line width=2pt,domain=0.0:9.5] plot(\x,{(-0.--3.141592653589793*\x)/3.141592653589794});
\draw [line width=2pt,domain=0.0:9.5] plot(\x,{(-0.-3.141592653589793*\x)/3.141592653589793});
\draw [line width=2.pt,color=black] (4.635238686543314,-4.635238686543314)-- (9.283185307179584,0.016814692820412837);
\draw [line width=2.pt,color=black] (9.283185307179584,0.016814692820412837)--(4.635238686543314,4.635238686543314) ;
\draw [->][line width=2.5pt] (3.141592653589793,-3.141592653589793)-- (4.635238686543314,-4.635238686543314);
\draw [->][line width=2.5pt] (4.635238686543314,-4.635238686543314)--(3.141592653589793,-3.141592653589793) ;
\draw[color=black] (7,5) node[xscale=1, yscale=1] {$\Delta$};
\draw[color=black] (3.19,0.17) node {$D$};
\draw[color=black] (5,2.816911927995097) node {$D_a$};
\draw[color=black] (-0.3,-0.25) node[below] {$0$};
\draw[color=black] (6,-0.2) node[below right] {$\pi$};
\draw[color=black ] (2.3,-3.8) node[sloped, below right] {$a$};
\end{tikzpicture}
\caption{\label{figure3}The squares $D$, $D_a$ and the sector $\Delta$.}
\end{center}
\end{figure}

\underline{\bf{$\bullet$ Step 1 :}} In this step we prove the following claim:
\begin{equation}\label{claim1}
 f_{1, \varepsilon}/ P \in A^2\left(\Delta \setminus D_a \right)
\end{equation}
 (the case $f_{2, \varepsilon}/P \in A^2\left(\left(\pi - \Delta\right) \setminus \left(\pi-D_a \right) \right)$ follows in a similar fashion).\\
To do so, remark that there exists a constant $C_a >0$ such that for any $z \notin D_a$,  $\abs{z} + 1 \leq C_a d(z,\partial D)$. 
Using the triangular inequality, we have for any $k \in \left\lbrace 1, \, 2 \right\rbrace$ and $\varepsilon \in \lbrace \pm \rbrace$,
\begin{align*}
\forall z \notin D_a, \ |f_{k, \varepsilon}(z)|&\leq  \frac{1}{\pi} \int_{\gamma_{k,\varepsilon}} \left| \frac{f(u)}{u-z} \right| \abs{du} \\
&\leq \left\|f\right\|_{L^1(\partial D)} \frac{1}{\pi d(z,\partial D)}\\
& \leq \frac{C}{\abs{z}+1},
\end{align*}
where we have used that $L\log^+ L\subset L^1$ on a segment.

So, since $-2i\pi \notin \overline{\Delta \setminus D_a}$, we obtain
\begin{align*}
\int_{\Delta \setminus D_a} \abs{\frac{f_{1,\varepsilon}(z)}{p(z)}}^2 dA(z) 
&\le C \int_{\Delta \setminus D_a} \frac{dA(z)}{\abs{2i\pi+z}^2(1+\abs{z})^2} < +\infty.
\end{align*}
This proves claim \eqref{claim1}. \\

\noindent
\underline{\bf{$\bullet$ Step 2 :}} This step is more delicate and uses the
Cauchy (or Hilbert) transform and the inclusion $H^1(\D)\subset A^2(\D)$.

We  need to show the following claim
\begin{equation}\label{claim2}
f_{1, \varepsilon}/ P \in A^2\left(D_a \right)
\end{equation}
(and $f_{2, \varepsilon}/P \in A^2\left(\pi-D_a\right)$). It is enough to treat the case $f_{1,+}$, the others follow in a similar way. 

For $g \in L^1(\R)$, we denote by $\mathcal{C}g$ its Cauchy Transform defined by 
$$
\left(\mathcal{C}g\right)(z)= \frac{1}{i\pi} \int_\R \frac{g(t)}{t-z} dt, z\in \C^+:=\enstq{z \in \C}{\mathrm{Im}(z) >0}. 
$$
For more details on this operator, we refer to \cite{CMR}. 

We first explain briefly how to translate $f_{1,+}$ to $\mathcal{C}g$ for some suitable $g$. It 
is essentially rotating and translating the line through $\gamma_{1,+}$ to $\R$.
To be more explicite, let $\alpha_{1,+}: z \mapsto 1+\frac{\sqrt{2}}{\pi}e^{i\frac{3\pi}{4}}z$. This is a direct similarity transformation which sends $D_a$ to $\C^+$ and in particular
$\gamma_{1,+}$ onto $[0,1]$ (note that the orientation is preserved, e.g.\ the endpoint 0 of $\gamma_{1,+}$ is sent to 1). Let ${f}^{\gamma}_{1,+}(t)=\mathds{1}_{\left[0, 1\right]}(t) f(\gamma_{1,+}(t))$. For all $z \in D_a$, we have 
\begin{align*}
f_{1, +}(z)=  \frac{1}{i\pi}\int_{\gamma_{1,+}} \frac{f(u)}{u-z} du 
&= \frac{1}{i\pi} \int_0^{1} \frac{f(\gamma_{1,+}(t))}{\gamma_{1,+}(t)-z}\gamma_{1,+}'(t) dt\\ 
&= \frac{1}{i\pi} \int_\R \frac{{f}^{\gamma}_{1,+}(t)}{(1-t)\frac{\pi}{2}(1+i)-z}\left(-\frac{\pi}{2}(1+i)\right) dt  \\
&= \frac{1}{i\pi} \int_\R \frac{{f}^{\gamma}_{1,+}(t)}{t-\alpha_{1,+}(z)} dt \\
&= \left(\mathcal{C} {f}^{\gamma}_{1,+}\right)\left(\alpha_{1,+}(z)\right)
\end{align*}
So, since $P$ does not vanish on $\overline{D_a}$, we obtain 
\begin{align}\label{estim3}
\int_{D_a} \abs{\frac{f_{1,+}(z)}{P(z)}}^2 dA(z) &\leq C\int_{D_a} \abs{{f_{1,+}(z)}}^2 dA(z)\nonumber\\
&=C \int_{D_a} \abs{\left(\mathcal{C} {f}^{\gamma}_{1,+}\right)(\alpha_{1,+}(z))}^2 dA(z) \nonumber\\
&=C_2 \int_{\alpha_{1,+}\left( D_a \right)} \abs{\left(\mathcal{C} {f}^{\gamma}_{1,+}\right)(z)}^2 dA(z),
\end{align}
where we have used in the last step that $\alpha_{1,+}$ is an affine change of variable with constant jacobian.
As already written, $\alpha_{1,+}\left( D_a \right)$ is a square in the upper-half plane with a segment of the real line as one of its sides. We will next appeal to the following regularity result of the Cauchy transform
which is essentially a combination of a result by Calderon-Zygmund \cite[Thm 2, p.100]{CZ} on the boundedness of the Cauchy transform for compactly supported functions from $L\log L$ to $H^1$, and a result by Hardy-Littlewood on inclusion between $H^1$ and $A^2$ on the disk.

\begin{prop}
\label{prop1}
Let $f \in L\log^{+}\!L(\R)$ have compact support. Let $\Omega$ be a square in the upper-half plane one side of which is a segment $I \subset \R$. 
Then the Cauchy transform $\mathcal{C}f$ belongs to $A^2(\Omega)$. 
\end{prop}

\begin{rmk}
Note that in this proposition, we do not need to assume any link between the (compact) support of $f$ and the segment $I$. However, we will apply later on the result for the case when the support of $f$ is included in $I$ (and $I=[0,1]$).
\end{rmk}


We start with a first intermediate result.


\begin{lemme}\label{LemInt1}
Let $f\in L\log^+ L(\R)$ having compact support. Then
the Cauchy Transform $\mathcal{C} f$ satisfies:
$$
\sup_{y>0}\  \int_{-L}^L |\mathcal{C}f(x+iy)|dx < +\infty
$$
\end{lemme}

This result is certainly known to the experts in harmonic analysis. We include its proof for convenience of the reader. It is essentially based on the following theorem by Calderon and Zygmund  \cite[Thm 2 p.100]{CZ}. Let $\tilde{f}_{\lambda}(x)=\int_{|x-y|>1/\lambda} f(y)/(x-y)dy$. Note that 
$\lim_{\lambda\to 0}\tilde{f}_{\lambda}(x)$ corresponds to the Hilbert transform of $f$.

\begin{thm}[Calderon-Zygmund]\label{CZthm}
If $|f|(1+\log_+|f|)$ is integrable over $\R$, then $\tilde{f}_{\lambda}$ is integrable over every set $S$ of finite measure. Moreover,
$$
 \int_S |\tilde{f}_{\lambda}|dx\le A_S\int_{\R}|f|(1+\log_+|f|)dx+B_S,
$$
where $A_S$ and $B_S$ are constants depending only on $S$, but neither on $f$ nor on $\lambda$.
\end{thm}

\begin{proof}[Proof of Lemma \ref{LemInt1}]
For $y>0$,
let $P_y(x)= \frac{y}{\pi(x^2 + y^2)}$ and $Q_y(x) = \frac{x}{\pi(x^2 + y^2)}$ be the Poisson  and the conjugate Poisson kernels ($Q_0$ corresponds to the kernel of the Hilbert transform). Then we have
$$ \forall z \in \C^+, \ \mathcal{C}f(z) = \mathcal{P}f(z)+ i \mathcal{Q}f(z)$$
where we have written $\mathcal{P}f(x+iy) = (P_y * f)(x)$ and $\mathcal{Q}f(x+iy)=(Q_y * f)(x)$. So it suffices to show
$$ \sup_{y>0}\  \int_{-L}^L |\mathcal{P}f(x+iy)|dx < +\infty \quad \text{and} \quad \sup_{y>0}\  \int_{-L}^L |\mathcal{Q}f(x+iy)|dx < +\infty. $$
The first inequality is clear from classical properties of the Poisson kernel (for this it is
even enough that $f\in L^1(\R)$, see \cite[Thm 3.1]{Ga}). 
Consider the second inequality. Recall the following estimate (see for example \cite[p. 105]{Ga})
$$ \forall y >0, \forall x \in \R, \ \abs{\mathcal{Q}_yf(x+iy) - \widetilde{f}_y(x)} \leq C Mf(x)$$ 
where $Mf$ is the Hardy-Littlewood Maximal function and $\widetilde{f_y}$ is defined by 
$$\forall x \in \R, \ \widetilde{f}_y(x) = \int_{\abs{t-x} > y} Q_y(x-t) f(t) dt.$$ 
This, together with a classical result on the regularity of $Mf$(see \cite[p. 23]{Ga}) 
and Theorem \ref{CZthm} above yields the desired result.
\end{proof}

We need some more notation.
Let $L >0$ such that $I \subset \left]-\frac{L}{2}, \frac{L}{2}\right[$ and $\supp f \subset \left]-\frac{L}{2},\frac{L}{2}\right[$. Denote by $\Omega_L$ the square contained in $\C^+$ with one side being the segment $[-L, L]$. 

\begin{lemme}\label{LemNouv}
Under the conditions of the proposition, we have $\mathcal{C}f \in E^1(\Omega_L)$.
\end{lemme}

\begin{proof}[Proof of Lemma \ref{LemNouv}]
In order to prove $\mathcal{C}f \in E^1(\Omega_L)$, pick $(\omega_\varepsilon)_{0< \varepsilon <\varepsilon_0 } $, $\varepsilon_0 < L/2$, a sequence of rectifiable Jordan curves given by the sides of the squares contained in $\Omega_L$  one side  of which is 
$\omega_{\varepsilon,0} = [-L+\varepsilon, L-\varepsilon] + i \varepsilon$ and 
$\omega_{\varepsilon,1}$ corresponds to the remaining three sides of the square (see Figure \ref{figure2}). Let $\omega_\varepsilon = \omega_{\varepsilon,0} \vee \omega_{\varepsilon, 1}$
(concatenation of the two Jordan curves, orientated counterclockwise). 
Then for any $0 < \varepsilon \leq \varepsilon_0$, we have
$ d(\omega_{\varepsilon, 1}, \supp f) >0$. 
Thus, from the very definition of the Cauchy transform and triangular inequality,
$$ \sup_{0 < \varepsilon < \varepsilon_0} \int_{\omega_{\varepsilon, 1}} |\mathcal{C}f| |dz| < + \infty. $$ 
It remains to show that
$$ \sup_{0 < \varepsilon < \varepsilon_0} \int_{\omega_{\varepsilon, 0}} |\mathcal{C}f| |dz| = \sup_{0 < \varepsilon < \varepsilon_0} \int_{-L+\varepsilon}^{L-\epsilon} |\mathcal{C}f(t+i\varepsilon)| dt < + \infty.$$
Using Lemma \ref{LemInt1}, we conclude that $\mathcal{C}f \in E^1(\Omega_L)$ 
\end{proof}


As mentioned above, the other ingredient in the proof of Proposition \ref{prop1} is the following interesting result due to Hardy and Littlewood (see \cite[thm31]{HL}, see also \cite{Vu} or \cite[Thm4.11, p. 282]{QQ} for a more elementary proof).
\begin{thm}[Hardy-Littlewood]
\label{HL}
The Hardy space $H^1(\D)$ embeds continuously into $A^2(\D)$. 
\end{thm}

We are now in a position to prove Proposition \ref{prop1}.

\begin{proof}[Proof of Proposition \ref{prop1}]
In view of Lemma \ref{LemNouv}, we already know that $\mathcal{C}f \in E^1(\Omega_L)$.
Now, if $\varphi : \D \to \Omega_L$ is a conformal mapping, then we will have $\left(\mathcal{C}f \circ \varphi \right)  \varphi' \in H^1(\D)$. From Theorem \cite{HL} we obtain
$\left(\mathcal{C}f \circ \varphi \right) \varphi' \in A^2(\D)$, or equivalently, by simple change of variable, $\mathcal{C}f \in A^2(\Omega_L)$. Since $\Omega \subset \Omega_L$, we obtain $\mathcal{C}f \in A^2(\Omega)$ which is what we want to prove.
\end{proof}

From the preceding discussions we can now deduce the claim \eqref{claim2}. Indeed, 
recall from \eqref{estim3} that
$$
\int_{D_a} \abs{\frac{f_{1,+}(z)}{P(z)}}^2 dA(z) \leq C\int_{\alpha_{1,+}\left( D_a \right)} \abs{\left(\mathcal{C} {f}^{\gamma}_{1,+}\right)(z)}^2 dA(z).
$$
Clearly, when $f\in L\log L$ with compact support, the same will be true for $f_{1,+}^{\gamma}$
(which is essentially a truncation of $f$ composed with a rotation/translation). From Proposition
\ref{prop1} (with  $\Omega=\alpha_{1,+}\left( D_a \right)$ being a unit square in the upper half plane with base 
on the real line we deduce \eqref{claim2} (the argument is the same for $f_{1,-}$).

\begin{proof}[Proof of Theorem \ref{thm1}]
By \eqref{eq2} it is enough to show that $f/P\in A^2(\Delta)+A^2(\pi-\Delta)$.
The decomposition will be given by $F_1=(f_{1,+}+f_{1,-})/P$ and $F_2=(f_{2,+}+f_{2,-})/P$.
By \eqref{claim1} we have $F_1\in A^2(\Delta\setminus D_a)$, and \eqref{claim2} implies
that $F_1\in A^2(D_A)$. The case $F_2$ is treated in exactly the same way.
\end{proof}

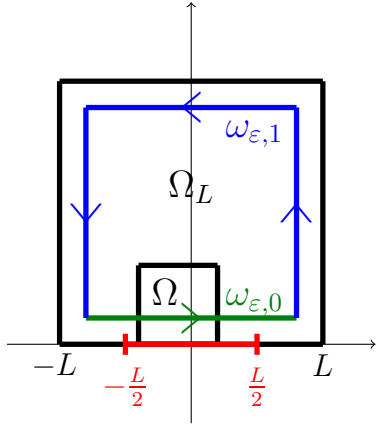
\begin{figure}[h]

\begin{center}
\definecolor{zzttqq}{rgb}{0.,0.5,0.05}
\begin{tikzpicture}[ scale=1.4]
\draw  [->](-1.75,0)-- (1.75,0);
\draw  [->](0,-0.75)-- (0,3.25);

\draw [line width=2.pt] (-1.25,0.)-- (1.25,0.);
\draw [line width=2.pt] (1.25,0.)-- (1.25,2.5);
\draw [line width=2.pt] (1.25,2.5)-- (-1.25,2.5);
\draw [line width=2.pt] (-1.25,2.5)-- (-1.25,0.);

\draw [line width=2.pt] (-0.5,0.)-- (0.25,0.);
\draw [line width=2.pt] (0.25,0.)-- (0.25,0.75);
\draw [line width=2.pt] (0.25,0.75)-- (-0.5,0.75);
\draw [line width=2.pt] (-0.5,0.75)-- (-0.5,0.);

\draw [line width=2.pt, color=zzttqq] (-1,0.25)-- (1,0.25)node[midway,sloped,xscale=1, yscale=2]{$>$};
\draw [line width=2.pt, color=blue] (1,0.25)-- (1,2.25) node[midway,sloped,xscale=1, yscale=2]{$>$};
\draw [line width=2.pt, color=blue] (1,2.25)-- (-1,2.25) node[midway,sloped,xscale=1, yscale=2]{$<$};
\draw [line width=2.pt, color=blue] (-1,2.25)-- (-1,0.25) node[midway,sloped,xscale=1, yscale=2]{$>$};

\draw [line width=2.pt, color=red] (-0.625,0)-- (0.625,0.);
\draw [line width=2.pt, color=red] (-0.625,0.1)-- (-0.625,-0.1)node[below] {\normalsize{$-\frac{L}{2}$}};
\draw [line width=2.pt, color=red] (0.625,0.1)-- (0.625,-0.1) node[below] {\normalsize{$\frac{L}{2}$}};

\begin{scriptsize};
\draw (-0.25,0.5) node {{\large{$\Omega$}}};
\draw (0.,1.5) node {{\large{$\Omega_L$}}};
\draw (0.6,0.4) node[color=zzttqq] {{\large{$\omega_{\varepsilon,0}$}}};
\draw (0.6,2) node[color=blue] {{\large{$\omega_{\varepsilon,1}$}}};
\draw (-1.3, -0.2) node {{\normalsize{$-L$}}};
\draw (1.25, -0.2) node {{\normalsize{$L$}}};
\end{scriptsize}
\end{tikzpicture}
\end{center}
\caption{The squares $\Omega$ et $\Omega_L$, and the path $\omega_\varepsilon$.} 
\label{figure2}
\end{figure}

\section{Proof of Theorem \ref{thm4} \label{subsection3.2}}

As already mentioned, we have to solve the First Cousin Problem for Bergman spaces.
We refer the reader to \cite[Thm 9.4.1]{AM} where some of the ideas used below can be found.
The first step into this direction is to construct a partition of unity associated with 
$$
 \Omega=\Delta\cup (\pi-\Delta)=:\Omega_1\cup\Omega_2.
$$
We are thus seeking for positive smooth functions $\varphi_i$, $i=1,2$, such that 
$\supp(\varphi_i)\subset \Omega_i$, $\varphi_1+\varphi_2=1$ on $\Omega$. The existence
of such a partition is of course a general fact. However, for the convenience
of the reader, we give a more explicite construction of $\varphi_i$ which also allows to 
see the optimality of the weight of the Bergman space appearing in Theorem
\ref{thm4}, at least for our method.

\begin{lemme}
\label{lemma1}
Let $z_0=\frac{\pi}{2}+i \frac{\pi}{2}$ and $z_1=\frac{\pi}{2}-i \frac{\pi}{2}$, the upper and lower vertices of $D$. 
Then there exists $\chi_1, \chi_2 \in C^{\infty}(\Omega)$ such that, for $i \in \{1, \, 2\}$, 
\begin{enumerate}
\item $0 \leq \chi_i \leq 1, \ \mathrm{supp}\chi_i \subset \Omega_i $
\item $\chi_1 + \chi_2 \equiv 1$ on $\Omega$
\item $\forall z \in D, \ \left\lvert\frac{\partial \chi_1}{\partial \bar{z}}(z) \right\rvert \leq \frac{C}{|z-z_0||z-z_1|}$
\end{enumerate}
\end{lemme}
\begin{proof}
Let $\psi \in C^{\infty}(\R)$ be a cutoff function such that $0\leq \psi \leq 1$, $\psi|_{(-\infty, \, 0]}=0$, $\psi|_{[1, +\infty)}=1$. It is clear that $\psi'$ is bounded on $\R$. 
We will now create an explicit support for $\chi_1$, smaller than $\Delta\setminus D$.
For that, let $\alpha :(-\frac{\pi}{2}, \frac{\pi}{2}) \ni y \mapsto \frac{-1}{\pi}(y-\frac{\pi}{2})(y+\frac{\pi}{2})$ and draw the curves of equations $C_1:x=\alpha(y) + \frac{\pi}{2}$ and $C_2:x=-\alpha(y)+\frac{\pi}{2}$. They are symmetric with respect to the line $x=\frac{\pi}{2}$ and are included in $D$ (see Figure \ref{figure4}). Now, define $\chi_1$ by 
$$\chi_1(x,y)= \psi\left(\frac{x+\alpha(y)- \frac{\pi}{2}}{2\alpha(y)}\right), \quad x, y \in \left(-\frac{\pi}{2}, \frac{\pi}{2}\right)$$ and complete it by $1$ on $\Delta\setminus D$ and
by $0$ on $(\pi-\Delta)\setminus D$. 
Let now $(x,y)\in D$. We set $E$ for the subset of $(x,y)\in D$ contained between
the two curves $C_1$ and $C_2$.
Clearly, when $x<-\alpha(y)+\frac{\pi}{2}$, then $\chi_1(x,y)=0$ on the left half of $D\setminus E$.
Also, when 
$x>\alpha(y)+\frac{\pi}{2}$, then, since $\alpha(y)>0$ on $(-\frac{\pi}{2},\frac{\pi}{2})$,
$$
 \frac{x+\alpha(y)- \frac{\pi}{2}}{2\alpha(y)}>1,
$$
and hence $\chi_1(x,y)=1$ on the right half of $D\setminus E$. This implies that the function $\chi_1$
is $C^{\infty}(\Omega)$, it takes values between $0$ and $1$, is 1 on $\Delta\setminus D$ (actually on $\Delta\setminus E$)
and $0$ on $(\pi-\Delta)\setminus D$ (actually on $(\pi-\Delta)\setminus E$).
If we write $\chi_2= 1-\chi_1$, the points 1. and 2. of the lemma are verified. To obtain the last point, observe that outside $E$ the derivatives of $\chi_1$ vanish and if $(x,y)$ belongs to $E$, we have $|x- \frac{\pi}{2}| < |\alpha(y)|$. The point 3. follows. 
\end{proof}

It can easily be seen from the mean value theorem that we cannot hope for a better estimate
than 3.\ in Lemma \ref{lemma1}, see Remark \ref{rem3.8} below.
\\

We are now in a position to prove Theorem \ref{thm4}. Beforehand, we need to introduce an
auxiliary function $P(z)=1+z^2$. Observe that multiplication by $P$ is an isomorphic
operation on $A^2\left(D, |(z-z_0)(z-z_1)|^{-2}\right)$.
Pick $\varphi\in A^2\left(D, |(z-z_0)(z-z_1)|^{-2}\right)$, and set $\varphi_0=\varphi P$.
Consider the functions
\begin{align*} 
h_1&= \chi_2 \varphi_0\quad \text{on } \Omega_1 \\
h_2 &= - \chi_1 \varphi_0 \quad \text{on } \Omega_2
\end{align*}
They satisfy $h_i \in C^{\infty} (\Omega_i)$ and $\varphi_0=h_1-h_2$ on $D$. To conclude we need to solve a $\overline{\partial}$-problem. For that, note that 
$$
 \frac{\partial h_1}{\partial \bar{z}} = \frac{\partial h_2}{\partial \bar{z}}  = -\varphi_0 \frac{\partial \chi_1}{\partial \bar{z}} \quad \text{on} \ D. 
$$
So we can define a function $v \in \C^{\infty}(\Omega)$ such that

\begin{eqnarray*}
 v = 
\begin{cases} \displaystyle \frac{\partial h_1}{\partial \bar{z}} \text{ on }\Omega_1,\\
  \displaystyle \frac{\partial h_2}{\partial \bar{z}} \text{ on }\Omega_2. 
\end{cases}
\end{eqnarray*}
We need the following Hörmander $L^2$-estimates for the $\bar{\partial}$-equation  \cite[Thm 4.2.1]{Ho}. 
\begin{thm}
Let $U$ be a domain in $\C$ 
and $a >0$. If $f \in L^2_{\mathrm{loc}}(U)$ and 
$$\int_U \abs{f(z)}^2 
(1+\abs{z}^2)^{2-a} dA(z) < + \infty$$ 
then there exists $u \in L^2_{\mathrm{loc}}(U)$ which solves the equation $\frac{\partial u}{\partial \bar{z}} = f$ on $U$ and such that 
$$a\int_U \abs{u(z)}^2 
(1+\abs{z}^2)^{-a} dA(z) \leq \int_U \abs{f(z)}^2 
(1+\abs{z}^2)^{2-a} dA(z).$$
\end{thm}
We now apply this theorem with $U=\Omega$, $f=v$ and $a=2$. More precisely, since
$\varphi_0\in A^2\left(D, |(z-z_0)(z-z_1)|^{-2}\right)$ and in view of condition 3.\ of Lemma \ref{lemma1}, we see that $v\in L^2(\Omega)$. Then H\"ormander's theorem yields a function $u\in L^2(\Omega,(1+|z|^2)^{-2})$ solving the $\overline{\partial}$-problem. Observe that
this also implies that $u/P\in L^2(\Omega)$. Define now $\varphi_1=h_1-u$ on $\Omega_1$ and $\varphi_2=h_2-u$ on $\Omega_2$. These are holomorphic functions since
$$
 \frac{\partial \varphi_i}{\partial\overline{z}}=\frac{\partial h_i}{\partial\overline{z}}-\frac{\partial u}{\partial\overline{z}}
 =v-\frac{\partial u}{\partial\overline{z}}=0, \quad \text{on}\ \Omega_i.
$$

Now on $D$,
$$
 P{\varphi}=\varphi_0=\varphi_1 - \varphi_2,
$$
so that
$$
 \varphi=\frac{\varphi_1}{P}-\frac{\varphi_2}{P}=\frac{h_1-u}{P}-\frac{h_2-u}{P}=:f_1+f_2,
$$
where the function $f_i=\frac{h_i-u}{P}$ are holomorphic on $\Omega_i$.
Recall also that 
$h_i\in L^2(D)$ and $h_i$ extends trivially to $\Omega_i\setminus D$, so that $h_i/P\in 
L^2(\Omega_i)$. Since $u/P\in L^2(\Omega)$, we thus get $f_i\in A^2(\Omega_i)$
which completes the proof of Theorem \ref{thm4}
\qedsymbol

\begin{figure}[h]
\begin{center}
\begin{tikzpicture}[scale=0.7]
\definecolor{ptc}{rgb}{0,0.3,0.9}
\definecolor{gc}{rgb}{0.53,0.27,0.73}
\definecolor{sec}{rgb}{0.87, 0.27,0.4}
\definecolor{curves}{rgb}{0.95, 0.1, 0.1}
\fill[line width=2.pt,fill=blue, fill opacity=0.5] (3.141592653589793-4,4) -- (3.141592653589793,0) -- (3.141592653589793-4,-4)--cycle;
\fill[line width=2.pt,fill=sec, fill opacity=0.5] (4,4) -- (0,0) -- (4,-4)--cycle;
\draw [samples=50,rotate around={-270.:(2.1876466018629817,0.)},xshift=2.1876466018629817cm,yshift=0.cm,line width=2.pt,domain=-1.570796326794897:1.570796326794897, color=curves] plot (\x,{(\x)^2/2/2.0});
\draw [samples=50,rotate around={-90.:(0.9539460517268117,0.)},xshift=0.9539460517268117cm,yshift=0.cm,line width=2.pt,domain=-1.570796326794897:1.570796326794897, color=curves] plot (\x,{(\x)^2/2/2.0});
\draw [line width=1.pt,color=black,domain=0:4, color=black] plot(\x,{(\x)});
\draw [line width=1.pt,color=black,domain=0:4, color=black] plot(\x,{(-\x)});
\draw [line width=1.pt,color=black,domain=-(4-3.141592653589793):3.141592653589793] plot(\x,{3.141592653589793 -\x});
\draw [line width=1.pt,color=black,domain=-(4-3.141592653589793):3.141592653589793] plot(\x,{\x - 3.141592653589793});
\draw [line width=0.5pt,domain=-4.3:7.3] plot(\x,{(--1.5707963267948966-0.*\x)/1.});
\draw [line width=1.pt,domain=-4.3:7.3] plot(\x,{(-1.5707963267948966-0.*\x)/1.});
\draw[color=black] (3.5,2) node[xscale=1, yscale=1] {$1$};
\draw[color=black] (3.5,0) node[xscale=1, yscale=1] {$1$};
\draw[color=black] (3.5,-2) node[xscale=1, yscale=1] {$1$};
\draw[color=black] (-0.5,2) node[xscale=1, yscale=1] {$0$};
\draw[color=black] (-0.5,0) node[xscale=1, yscale=1] {$0$};
\draw[color=black] (-0.5,-2) node[xscale=1, yscale=1] {$0$};
\draw[color=black] (0.6,0) node[xscale=1, yscale=1] {$0$};
\draw[color=black] (2.5,0) node[xscale=1, yscale=1] {$1$};
\end{tikzpicture}
\caption{\label{figure4} Values of $\chi_1$ on $\Omega$, and
the curves $x=\pm \alpha(y) + \pi/2$.}
\end{center}
\end{figure}
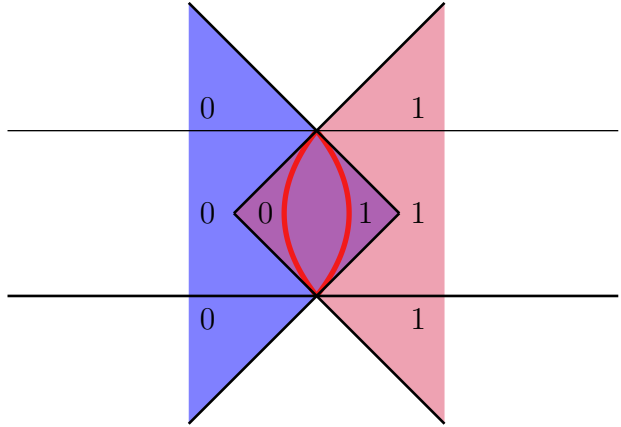

\begin{rmk}\label{rem3.8}
Condition 3. of Lemma \ref{lemma1} is optimal in the sense that 
$$\limsup_{z\in D,z\to z_i}
\left\lvert(z-z_i)\frac{\partial \chi_1}{\partial \bar{z}}(z) \right\rvert>0,\quad i=0,1,
$$ 
for any arbitrary  partition of unity for $\{\Omega_1,\Omega_2\}$.
This is an easy consequence of the mean value theorem.
So, it does not seem possible to improve the result examinating further the First Cousin
Problem here. 


\end{rmk}

\bibliographystyle{alpha} 
\bibliography{biblio} 

\end{document}